\documentclass[12pt]{amsart}

\newtheorem{theorem}{Theorem}[section]

\newtheorem{proposition}[theorem]{Proposition}

\newtheorem{corollary}[theorem]{Corollary}

\newtheorem{remark}[theorem]{Remark}

\newtheorem{lemma}[theorem]{Lemma}

\newtheorem{example}[theorem]{Example}

\newtheorem{conjecture}[theorem]{Conjecture}
\newcommand{\lra}{\longrightarrow}

\newcommand{\Cliff}{\operatorname{Cliff}}
\newcommand{\Ext}{\operatorname{Ext}}
\newcommand{\Hom}{\operatorname{Hom}}
\newcommand{\Grass}{\operatorname{Grass}}
\newcommand{\gr}{\operatorname{gr}}

\usepackage{amssymb}

\numberwithin{equation}{section}

\begin{document}

\title{On generated coherent systems and a conjecture of D. C. Butler}


\date{\today}
\thanks{All authors are members of the research group VBAC (Vector Bundles on Algebraic Curves). The authors were also participants in the programme ``Moduli Spaces'' held at the Newton Institute in 2011, where work on this paper was begun. The first author would like to thank ICTP for hospitality during the preparation of the paper and CONACYT for support. }

\author{L. Brambila-Paz}

\address{CIMAT, Mineral de Valenciana S/N, Apdo. Postal 402, C.P. 36240. Guanajuato, Gto,
M\'exico}
\email{lebp@cimat.mx}

\author{O. Mata-Gutierrez}

\address{Departamento de Matem\'aticas, CUCEI, Universidad de Guadalajara. Av. Revoluci\'on 1500. C.P. 44430, Guadalajara, Jalisco,
M\'exico} \email{osbaldo.mata@academico.udg.mx, osbaldo@cimat.mx}

\author{P. E. Newstead}

\address{Department of Mathematical Sciences\\
University of Liverpool\\
Peach Street, Liverpool, L69 7ZL}

\email{newstead@liverpool.ac.uk}

\author{Angela Ortega}

\address{Institut f\"ur Mathematik\\
        Universit\"at Humboldt zu Berlin\\
Unter den Linden 6, D-10099 Berlin,Germany}

\email{ortega@math.hu-berlin.de}

\begin{abstract}   Let $(E,V)$ be a general generated coherent system of type $(n,d,n+m)$ on a general non-singular irreducible complex projective curve. A conjecture of D. C. Butler relates the semistability of $E$ to the semistability of the kernel of the evaluation map $V\otimes \mathcal{O}_X\to E$. The aim of this paper is to obtain results on the existence of generated coherent systems and use them to prove Butler's Conjecture in some cases. The strongest results are obtained for type $(2,d,4)$, which is the first previously unknown case.
\end{abstract}
\maketitle

\section{Introduction}\label{intro}

Let $X$  be a non-singular irreducible complex projective curve of genus $g$. Morphisms from $X$ to a Grassmannian are of interest in studying the geometry of $X$ and, in particular, its syzygies (see \cite{an,f} for surveys of results on syzygies). In more detail, let $\varphi:X\to\Grass(m,n+m)$ be a morphism from $X$ to the Grassmannian of $m$-dimensional subspaces of $\mathbb{C}^{n+m}$ and let $0\to S\to\mathcal{O}_G^{n+m}\to Q\to0$ be the tautological sequence on $\Grass(m,n+m)$. Suppose that $\varphi$ is non-degenerate in the sense that neither $\varphi^*(S)$ nor $\varphi^*(Q)$ admits $\mathcal{O}_X$ as a direct factor. The pullback of the tautological sequence then gives rise to an exact sequence which we can write as
\begin{equation}\label{eq0}
0\lra M_{V,E} \lra V\otimes \mathcal{O}_X\stackrel{ev}{\lra} E\lra 0.
\end{equation}
The non-degeneracy condition implies that $h^0(M_{V,E})=0$ (and also $h^0(E^*)=0$), so that $V$ can be regarded as a linear subspace of $H^0(E)$ and $(E,V)$ is a coherent system on $X$ of type $(n,d,n+m)$ for some $d$. Moreover $(E,V)$ is generated in the sense that the evaluation map $V\otimes \mathcal{O}_X\stackrel{ev}{\lra} E$ is surjective. Conversely, any such sequence \eqref{eq0} gives rise to a non-degenerate morphism $\varphi: X\to\Grass(m,n+m)$ which is uniquely determined up to the action of ${\rm GL}(n+m)$ on $\Grass(m,n+m)$. Thus the study of such morphisms is reduced to the study of generated coherent systems $(E,V)$ for which $h^0(E^*)=0$. References for work involving this correspondence include \cite{ball,balr,fo,gmn,mon2,T3}.

 In this paper, we study generated coherent systems $(E,V)$ of type $(n,d,n+m)$, in particular those that are $\alpha $-stable for small $\alpha >0$. Recall more generally that a coherent system of type $(n,d,k)$ on $X$ is a pair $(E,V)$, where $E$ is a vector bundle on $X$ of rank $n$ and degree $d$ and $V\subset H^0(X,E)$ is a linear subspace of dimension $k$. Denote by $G(\alpha;n,d,k)$ the moduli space of $\alpha$-stable coherent
systems of type $(n,d,k)$ (for the definition and properties of $\alpha$-stability, see Section \ref{seccs}). In particular, we denote the moduli space by $G_{0}(n,d,k)$  when $\alpha >0$ is small and by $G_{L}(n,d,k)$ when  $\alpha >0$  is large. We define
$$S(\alpha;n,d,n+m):=\{ (E,V)\in G(\alpha;n,d,n+m)\,|\, V\otimes \mathcal{O}_X\stackrel{ev}{\lra} E\mbox{ is surjective}\}.$$
We write also $S_0(n,d,n+m)$ for $S(\alpha;n,d,n+m)$ when $\alpha$ is small, $M(n,d)$ for the moduli space of stable bundles of rank $n$ and degree $d$, and $B(n,d,k)$  for the Brill-Noether locus consisting of those $E\in M(n,d)$ for which $h^0(E)\ge k$.

 Our first aim is to study the non-emptiness and geometry (e.g. irreducibility, dimension, smoothness) of $S_0(n,d,n+m)$. While there are many results in the literature concerning the non-emptiness of $B(n,d,n+m)$ (and hence $G_0(n,d,n+m)$)(see, for example \cite{mon1} and, for $n=2$, \cite{cf1,cf2}) and indeed for $G(\alpha;n,d,n+m)$ for any $\alpha>0$ \cite{T2,mon}, much less is known about $S_0(n,d,n+m)$ except in the case $d\le2n$ \cite[Theorem 4.4(c)]{bgmmn2}.

In order to state our first general result for small degree, we recall that $\Cliff_n(X)$ is the rank-$n$ Clifford index of $X$ (see \eqref{eqcliff0}).

\begin{theorem}[\bf Theorem \ref{beta}] \
\begin{enumerate}
\item[(i)] If $d<\min\{2m,ng+m\}$, then $S_0(n,d,n+m)=\emptyset$.
\item[(ii)] If $X$ is a general curve and  $d< n+\frac{ng}{n+1}$ or $d< m+\frac{mg}{m+1}$, then $S_0(n,d,n+m)=\emptyset$ and $S_0(m,d,n+m)=\emptyset$. This holds in particular if $d<g+s$ and $g\le s$, where $s=\max\{n,m\}$.
\item[(iii)] If $X$ is a curve of genus $g\ge4$, $m\ge n$ and
$$d\le\min\{2m-1+n\Cliff_n(X),n(g-1)+m-n\},$$
then $S_0(n,d,n+m)=\emptyset$.
\end{enumerate}
\end{theorem}

For $d> n(2g-1)$, it is well known that $S_0(n,d,d+n(1-g))\ne \emptyset $.  We use the Picard sheaf $\mathcal{W}$ to show that, if $d\ge ng+m$ with $m\ge1$, then $S_0(n,d,n+m)\ne\emptyset$. In addition, we describe in Theorem \ref{props0} some geometrical properties of $S_0(n,d,n+m)$ in this case and prove that, if $m\ge ng$, then  $S_0(n,d,n+m)\ne\emptyset$ if and only if $d\ge ng+m$ (Corollary \ref{cor3.4}).

For any $(E,V)\in S_0(n,d,n+m)$, we have an exact sequence \eqref{eq0}.
By \cite[Lemma 2.9]{bgmmn}, $h^0(E^*)=0$, so, dualising \eqref{eq0}, we obtain a generated coherent system $(M^*_{V,E}, V^*)$, of type $(m,d,n+m)$. This is called the {\it dual span} of $(E,V)$ and denoted by $D(E,V)$. Note that the pullback of the tangent bundle of $\Grass(m,n+m)$ by the morphism $\varphi$ defined by \eqref{eq0}  is isomorphic to $E\otimes M^*_{V,E}$. It follows that,
if $E$ and $M^*_{V,E}$ are both semistable, then so is this pullback. The same bundle is given by the pullback of the tangent bundle of $\Grass(n,n+m)$ by the morphism defined by $D(E,V)$.

There is a natural isomorphism between $\Grass(n,n+m)$ and $\Grass(m,n+m)$, so one can expect a relation between $S_0(n,d,n+m)$ and $S_0(m,d,n+m)$.  D. C. Butler, inspired by his result \cite[Theorem 2.1]{bu} and by the papers \cite{rv} and \cite{kn}, conjectured in \cite{but} that, for a general curve $X$  of genus $g$ and $m\geq 1$, $S_{0}(n,d,n+m)$ is dense in $G_0(n,d,n+m)$ and

\begin{conjecture}\label{conj1}   Let $X$ be a general curve of genus $g$ and $n,d,m$ positive integers. Then, for a general $(E,V)\in S_{0}(n,d,n+m)$,
$D(E,V)\in S_{0}(m,d,n+m)$. Moreover,  $S_{0}(n,d,n+m)$ and $S_{0}(m,d,n+m)$ are birational.
\end{conjecture}
Actually, Butler stated his conjecture for $g\ge3$, but we state it without this restriction. By {\it  general}, we mean that $(E,V)$ belongs to a Zariski open set which is dense in $S_0(n,d,n+m)$. We will refer to Conjecture \ref{conj1} as {\it Butler's conjecture for $(n,d,n+m)$}. Note that the conjecture for $(n,d,n+m)$ holds in a trivial sense if $S_{0}(n,d,n+m)=\emptyset$ and $S_{0}(m,d,n+m)=\emptyset$. We shall say that {\it Butler's Conjecture holds non-trivially} for $(n,d,n+m)$ if Conjecture \ref{conj1} holds and $S_0(n,d,n+m)\ne\emptyset$. Note that Conjecture \ref{conj1} is closely connected with the stability of Picard bundles (see \cite{el,li}).

To proceed further, we introduce some definitions. We write
$$
U^s(n,d,n+m):=\{(E,V)\in G_0(n,d,n+m)\,|\,(E,V) \mbox{ is $\alpha$-stable for all }\alpha>0\},
$$
$$U_g(n,d,n+m):= \{(E,V)\in U^s(n,d,n+m)\,|\,(E,V) \mbox{ is generated}\},$$
$$T(n,d,n+m):=\{(E,V)\in S_0(n,d,n+m)\,|\,D(E,V)\in S_0(m,d,n+m)\},$$
and
$${\mathcal P}_0(n,d,n+m):=\{(E,V)\in S_0(n,d,n+m)\,|\,     (E,V)\mbox{ admits a subpencil}\},$$
(for the definition of subpencil, see Section \ref{secug}).

In Lemma  \ref{genopen}, we prove that $T(n,d,n+m)$ is open in $S_0(n,d,n+m)$; hence, the formula $(E,V)\mapsto D(E,V)$ defines an isomorphism
$$D:T(n,d,n+m)\to T(m,d,n+m)$$
(Proposition \ref{propt}), so that $T(n,d,n+m)$ is the largest open subscheme of $S_0(n,d,n+m)$ for which the conclusion of Butler's Conjecture holds. The next theorem gives sufficient conditions for Butler's conjecture to be fulfilled.

\begin{theorem}[\bf Theorem \ref{cort}] \  Butler's Conjecture holds for $(n,d,n+m)$ if and only if $T(n,d,n+m)$ is dense in $S_0(n,d,n+m)$ and $T(m,d,n+m)$ is dense in $S_0(m,d,n+m)$. In particular, if both $S_0(n,d,n+m)$ and $S_0(m,d,n+m)$ are irreducible and $T(n,d,n+m)$ is non-empty, then Butler's Conjecture holds non-trivially for $(n,d,n+m)$.
\end{theorem}

 Butler's Conjecture for $(1,d,m+1)$ was established for $g\ge1$ in \cite{bbn} following earlier work by many authors. In \cite[Theorem 2.1]{bu}, D. C. Butler  proved that, if $E$ is semistable of rank $n$ and degree $d\ge2ng$, then $M_{H^0(E),E}$ is semistable; moreover, if $E$ is stable and $d>2gn$, then $M_{H^0(E),E}$ is stable. For general $n$ and $d> 2ng$, this is the first step in proving the conjecture for $(n,d,d+n(1-g))$.
In  \cite[Ch. 2 Th\'{e}or\`{e}me B-1]{mer}, Mercat stated a result which completes the proof in this case, except that he does not fully consider the case where the underlying bundle is strictly semistable.

Using the above results we complete  Mercat's result  for $d>2ng$ and extend it to cover $d=2ng$. In the statement,  $\beta(n,d,k)$ is the ``expected dimension'' of $G(\alpha;n,d,k)$ (for further information, see Section \ref{seccs}). Note that, if $h^0(E)=k$ for all $E\in B(n,d,k)$, we can regard $B(n,d,k)$ as an open subset of $G_0(n,d,k)$.

\begin{theorem}[{\bf Theorem \ref{mercat}}]
Let $X$ be a curve of genus $g\ge2$ and let $m=d-ng$.
\begin{enumerate}
\item[(i)] If $d>2ng-n$, then
\begin{equation*}
S_0(n,d,n+m)=G_0(n,d,n+m)=M(n,d)
\end{equation*}
and, in particular, $S_0(n,d,n+m)$ is smooth and irreducible of dimension $\beta(n,d,n+m)$.
\item[(ii)] If $d>2ng$, then $S_0(m,d,n+m)=G_0(m,d,n+m)=B(m,d,n+m)$; moreover, Butler's Conjecture holds non-trivially and the dual span construction defines an isomorphism
\begin{equation*}
D:S_0(n,d,n+m)\lra S_0(m,d,n+m).
\end{equation*}
\item[(iii)] If $g\ge3$ and $X$ is not hyperelliptic, then $S_0(ng, 2ng,n(g+1))$ is irreducible and
\begin{equation*}
\emptyset\ne B(ng,2ng,n(g+1))\subset T(ng,2ng,n(g+1))\subset S_0(ng,2ng,n(g+1));
\end{equation*}
moreover, Butler's Conjecture holds non-trivially for $(n,2ng,n(g+1))$.
\end{enumerate}
\end{theorem}

The relation $T(n,d,n+m)\subset U_g(n,d,n+m)$ is established in Theorem \ref{main}. Thus, if $(E,V)\in S_0(n,d,n+m)$ is not $\alpha$-stable for all $\alpha >0$, $D(E,V)\notin  S_{0}(m,d,n+m)$. Moreover, we prove in Proposition \ref{coropencil} that, if $m\le n$ and $(E,V)\in \mathcal{P}_0(n,d,n+m)$ then $D(E,V)\notin  S_{0}(m,d,n+m)$. These results give us conditions under which Butler's Conjecture may fail.
For large enough degree, it can also be proved (Proposition \ref{dimcomp}) that, if $T(n,d,n+m)\ne\emptyset$, then it is smooth of the expected dimension $\beta(n,d,n+m)$, so that $U_g(n,d,n+m)$ and $U_g(m,d,n+m)$ have an irreducible component of dimension $\beta(n,d,k)$.

We use the above results to obtain results in the first particular case of interest, that is when $(n,d,n+m)=(2,d,4)$.  We show first that $U_g(2,d,4)=S_0(2,d,4)$ (Proposition \ref{proposition2}) and  more complete results of non-emptiness of $S_0(2,d,4)$ are given in Theorem \ref{p3.1}.  We prove several propositions concerning $\mathcal{P}_0(2,d,4)$, leading to the following theorem.

\begin{theorem}[{\bf Theorem \ref{cor3.2}}]
Let $X$ be a general curve of genus $g\ge3$. Then there exist integers $d$ such that
\begin{equation*}
\max\left\{g-r+\left\lfloor\frac{g}r\right\rfloor,g+r+3\right\}\le d\le2\left(\left\lceil\frac{2g}3\right\rceil+2\right)\mbox{ for some integer }r\ge1.
\end{equation*}
Moreover, if $d$ satisfies these inequalities, then
\begin{enumerate}
\item[(i)] $T(2,d,4)=S_0(2,d,4)\setminus \mathcal{P}_0(2,d,4)$;
\item[(ii)] $\mathcal{P}_0(2,d,4)\ne\emptyset$ and $\dim\mathcal{P}_0(2,d,4)<\beta(2,d,4)$;
\item[(iii)] Butler's Conjecture holds non-trivially for $(2,d,4)$.
\end{enumerate}
If, in addition, $d\ge2g-1$, then $\dim S_0(2,d,4)=\beta(2,d,4)$ and $T(2,d,4)$ is smooth.
\end{theorem}

 In the case of genus $6$, we can obtain more detailed information using the results of Section \ref{sec2d4} and some other techniques (Theorem \ref{p3.3}); this includes in particular a complete description for $d\le12$. The proof involves a number of special arguments in addition to those used in Section \ref{sec2d4}.
In addition to this theorem, we obtain results for coherent systems of type $(2,d,n+2)$ with $n\ge3$ on a curve of genus $6$ and, in particular, observe that Butler's Conjecture holds non-trivially for $(2,10,5)$ (see Proposition \ref{newst26}).

The paper is organized as follows. In Section \ref{seccs}, we give the main results on coherent systems that we will use. In Section \ref{secfirst} we obtain our main results on non-emptiness of $S_0(n,d,k)$ for general $g$  and prove Theorems \ref{cort} and \ref{mercat}. In Section \ref{secug}, we obtain some general results of $P_0(n,d,k)$, $T(n,d,k)$ and $U_g(n,d,k)$.
Section \ref{sec2d4} is devoted to the case $(2,d,4)$ and Section \ref{secg6} to genus $g=6$.

Notation: For a vector bundle $E$ over $X$, we denote by $d_E$, $n_E$, $\mu(E)$ and $h^i(E)$ the degree, the rank of $E$, the slope $\frac{d_E}{n_E}$ of $E$ and the dimension of $H^i(E)$, respectively. We say that a coherent system $(E,V)$ is generated if the evaluation map $V\otimes \mathcal{O}_X\to E$ is surjective and that $E$ is generated if $(E,H^0(E)$) is generated. Throughout the paper, $X$ will denote a non-singular irreducible projective curve defined over ${\mathbb C}$ and $K_X$ will be the canonical bundle on $X$. We make no assumptions about the genus or generality of $X$ except where stated.

\section{Review of coherent systems}\label{seccs}

In this section, we recall the main results on coherent systems that we will use. For a more complete treatment of the subject, see \cite{bgmn} and \cite{news} and the bibliographies therein.

Let $(E,V)$ be a coherent system {\it of type}
$(n,d,k)$ on $X$. A {\it subsystem} of $(E,V)$ is a coherent system $(F,W)$ such that
$F\subset E$ is a subbundle of $E$ and $W\subset H^0(F)\cap V$. For a real number $\alpha$, {\it the $\alpha $-slope} of a coherent system $(E,V)$ of type $(n,d,k)$, denoted by $\mu _{\alpha}(E,V)$, is the quotient
$$\mu _{\alpha}(E,V):= \frac{d+\alpha k}{n}.$$
A coherent system $(E,V)$ is {\it $\alpha$-stable} (resp. {\it $\alpha$-semistable}) if, for all proper subsystems $(F,W)$,
$$\mu _{\alpha}(F,W) <\mu _{\alpha}(E,V) \ \  \ \ (\mbox{resp.} \   \leq ). $$

There exist moduli spaces $G(\alpha;n,d,k)$ of $\alpha$-stable coherent systems of type $(n,d,k)$.
If $k\ge1$, a necessary condition for the non-emptiness of $G(\alpha;n,d,k)$ is
that $\alpha>0$. There are finitely many critical values
$0=\alpha_0<\alpha_1<\cdots<\alpha_L$ of $\alpha$; as $\alpha$
varies, the concept of $\alpha$-stability remains constant between
two consecutive critical values. We denote by $G_0(n,d,k)$ (resp.
$G_L(n,d,k)$) the moduli spaces corresponding to
$0<\alpha<\alpha_1$ (resp. $\alpha>\alpha_L$).
Let
\begin{equation}\label{eq21}
\beta (n,d,k):= n^2(g-1)+1-k(k-d+n(g-1))
\end{equation}
be the {\it Brill-Noether number} for $(g,n,d,k)$, often referred to as the {\it expected dimension}. In fact, every component of
$G(\alpha;n,d,k)$ has dimension $\geq \beta (n,d,k)$ and $G(\alpha;n,d,k)$ is smooth of dimension $\beta (n,d,k)$ in a neighbourhood of $(E,V)$
if and only if the Petri map
\begin{equation}\label{eqpetri}
V\otimes H^0(E^*\otimes K_X)\lra H^0(E\otimes E^*\otimes K_X )
\end{equation}
is injective. When $n=1$, the stability condition is vacuous and we write simply $G(1,d,n+m)$; this is the classical variety of linear systems.

Let $M(n,d)$ (resp.\ $\widetilde{M}(n,d)$) denote the moduli space of stable (resp. S-equivalence classes of semistable) bundles of rank $n$ and degree $d$ on $X$. The {\it Brill-Noether loci} are defined by
$$B(n,d,k):=\{E\in M(n,d)\,|\, h^0(E)\ge k\},\quad\widetilde{B}(n,d,k):=\{[E]\in\widetilde{M}(n,d)\,|\,h^0(\gr E)\ge k\},$$
where $[E]$ denotes the S-equivalence class of $E$ and $\gr E$ is the graded object associated with $E$ through a Jordan-H\"older filtration.

\begin{remark}\begin{em}\label{remvarios}
The following facts are well known.
\begin{enumerate}
\item  There is a forgetful morphism
$$G_0(n,d,k)\lra \widetilde{{B}}(n,d,k):(E,V)\mapsto [E].$$
\item If $E$ is stable, then,
for any linear subspace $V\subset H^0(E)$ of dimension $k$,  $(E,V)\in G_0(n,d,k)$.
\item If $(E,V)\in G_L(n,d,k)$ and $E$ is stable, then $(E,V)\in G(\alpha;n,d,k)$ for all $\alpha >0.$
\item Suppose $n\ge2$, $k\ge1$ and $\alpha>0$. If $(E,V)\in G(\alpha;n,d,k)$, then $d>0$ and $h^0(E^*)=0$.
\end{enumerate}
Here (1)--(3) follow easily from the definitions; for (4), see \cite[Lemma 2.9]{bgmmn}.

\end{em}\end{remark}

\begin{remark}\begin{em}\label{subnoalfa} If $(E,V)\in G_0(n,d,k)$ but $(E,V)\notin G_L(n,d,k)$, then there exists an $\alpha_i>0$ such that $(E,V)$ is strictly $\alpha_i$-semistable. In fact, there
exists an $\alpha_i$-semistable coherent subsystem $(F,W)$ of $(E,V)$ of type $(n',d',k')$ with $\mu_{\alpha_i}(F,W)=\mu_{\alpha_i}(E,V)$, such that $\frac{k}{n} < \frac{k'}{n'}$ and $\frac{d}n>\frac{d'}{n'}$ (see \cite[Lemma 6.5(ii)]{bgmn}).
\end{em}\end{remark}

In this paper, we are mainly interested in generated coherent systems. In this case, $k\ge n$ and we write $n+m$ in place of $k$. If $(E,V)$ is generated of type $(n,d,n+m)$ and $E\not\simeq \mathcal{O}_X^{\oplus n}$, then $d>0$ and $m\ge1$. By Remark \ref{remvarios}(4), if in addition $(E,V)$ is $\alpha$-stable for some $\alpha$, then $h^0(E^*)=0$. It is therefore of interest to consider coherent systems with this property.

\begin{remark}\begin{em}\label{rem1g} Let $(F,W)$ be a subsystem of a generated coherent system of type $(n,d,n+m)$ with $h^0(E^*)=0$. Since any quotient $Q$ of $E$ is generated and $h^0(Q^*)=0$, we have $h^0(Q)\ge n_Q+1$, so $\dim W\leq n_F+m-1$. In particular,
if $(F,W)$ is of type $(1,d',k')$, $k'\leq m$.
\end{em}\end{remark}

\begin{lemma}\label{rembounds}
Let $X$ be a general curve. If $(E,V)$ is a generated coherent system of type $(n,d,n+m)$ and $h^0(E^*)=0$, then
\begin{equation}\label{eqgen11}
d\geq n+\frac{ng}{n+1}
\end{equation}
and
\begin{equation}\label{eqgen21}
d\geq m+\frac{mg}{m+1}.
\end{equation}
\end{lemma}
\begin{proof}
Since $m\ge1$, we can choose a subspace $W\subset V$ with $\dim W=n+1$ which generates $E$.
This gives an exact sequence
\begin{equation}\label{eqgen1}
0\lra \det (E)^*\lra W\otimes \mathcal{O}_X\lra E\lra 0.
\end{equation}
Moreover, it is a standard fact that there exists an exact sequence
\begin{equation}\label{eqgen2}
0\lra \mathcal{O}_X^{n-1}\lra E\lra \det (E)\lra 0.
\end{equation}

From the dual of the exact sequence $(\ref{eqgen1})$ and the fact that $h^0(E^*)=0$, we have $h^0(\det (E))\geq n+1$; so, by classical Brill-Noether theory, $\beta(1,d,n+1)\ge0$, which is equivalent to \eqref{eqgen11}. The  cohomology sequence of (\ref{eqgen2}) gives $h^0(\det (E))\geq m+1$, hence, again by classical Brill-Noether theory, we obtain \eqref{eqgen21}.
\end{proof}

We can obtain a different bound by using higher rank Clifford indices. Recall first Clifford's Theorem for $\alpha$-semistable coherent systems (see \cite{nl}). This states that, for any $\alpha$-semistable coherent system $(E,V)$ of type $(n,d,k)$,
\begin{equation}\label{eqcliff1}
\frac{k}{n}\leq\begin{cases}
\mu(E) +1 -g &\mbox{if} \  \mu (E)\geq 2g\\
\frac{\mu(E)}{2}+1&\mbox{if} \  \mu (E)\leq 2g.\end{cases}
\end{equation}
 We recall now the definition of the rank-$n$ Clifford index from \cite{lnci} (where $\Cliff_n(X)$ is denoted by $\gamma'_n$). We write first, for any bundle $E$,
$$\gamma(E):=\frac1{n_E}(d_E-2(h^0(E)-n_E)).$$
Then, for any curve $X$ of genus $g\ge4$,
\begin{equation}\label{eqcliff0}
\Cliff_n(X):=\min\{\gamma(E)|E \mbox{ semistable},\  n_E=n,\  h^0(E)\ge2n,\  d_E\le n(g-1)\}.
\end{equation}
Any semistable bundle $E$ of rank $n$ with $h^0(E)\ge2n$ and $d_E\le n(g-1)$ is said to {\it contribute to} $\Cliff_n(X)$.

\begin{lemma}\label{lcliff}
Let $X$ be a curve of genus $g\ge4$. If $(E,V)\in G_0(n,d,n+m)$ with $m\ge n$, then
\begin{equation}\label{eqcliff}
d\ge \min\{2m+n\Cliff_n(X),n(g-1)+m-n+1\}.
\end{equation}
\end{lemma}

\begin{proof}
We can assume that $d\le n(g-1)+m-n$, so that either $E$ or $K_X\otimes E$ contributes to $\Cliff_n(X)$. The result now follows from the definition of $\Cliff_n(X)$.
\end{proof}

\begin{remark}\label{r26}\begin{em}
For $X$ general of genus $g\ge4$ and $(E,V)\in G_0(n,d,n+m)$ with $m\ge n$, Lemma \ref{lcliff} gives a stronger result than Lemma \ref{rembounds} provided
\begin{equation}\label{eqbound}
n\Cliff_n(X)\ge\frac{mg}{m+1}-m+1.
\end{equation}
This is certainly true if $n=2$, since $\Cliff_2(X)=\Cliff(X)=\lfloor\frac{g-1}2\rfloor$ \cite{bf}. For $n=3$, the lower bound for $\Cliff_3(X)$ in \cite[Theorem 4.1]{lnrk3} shows that \eqref{eqbound} holds for $g\le420$. Note, however, that this lower bound holds for any curve of genus $g\ge7$; the lower bound for a general curve is likely to be much larger. It is reasonable to conjecture that \eqref{eqbound} holds for the general curve of genus $g\ge4$ for all $n\ge2$ and $m\ge n$.
\end{em}\end{remark}

\section{$S_0(n,d,n+m)$ and Butler's conjecture}\label{secfirst}

From the results of Section \ref{seccs}, it follows that $S_0(n,d,n+m)=\emptyset$ for sufficiently small $d$. More precisely, we have the following theorem.

\begin{theorem}\label{beta} \
\begin{enumerate}
\item[(i)] If $d<\min\{2m,ng+m\}$, then $S_0(n,d,n+m)=\emptyset$.
\item[(ii)] If $X$ is a general curve and  $d< n+\frac{ng}{n+1}$ or $d< m+\frac{mg}{m+1}$, then $S_0(n,d,n+m)=\emptyset$ and $S_0(m,d,n+m)=\emptyset$. This holds in particular if $d<g+s$ and $g\le s$, where $s=\max\{n,m\}$.
\item[(iii)] If $X$ is a curve of genus $g\ge4$, $m\ge n$ and
$$d\le\min\{2m-1+n\Cliff_n(X),n(g-1)+m-n\},$$
then $S_0(n,d,n+m)=\emptyset$.
\end{enumerate}
\end{theorem}

\begin{proof} (i) follows from \eqref{eqcliff1}. (ii) follows from Lemma \ref{rembounds} (see also\cite[Theorem 3.9]{mio}). (iii) follows from Lemma \ref{lcliff}.
\end{proof}

The following lemma will play an important r\^ole in what follows.

\begin{lemma}\label{genopen}
$S(\alpha;n,d,n+m)$ is open in $G(\alpha;n,d,n+m)$. Moreover $U_g(n,d,n+m)$ is an open subset of $S(\alpha;n,d,n+m)$ for all $\alpha>0$ and  $T(n,d,n+m)$ is open in $S_0(n,d,n+m)$.
\end{lemma}

\begin{proof}
Suppose first that $\gcd(n,d,n+m)=1$. By \cite[Proposition A.8]{bgmmn}, there exists a universal family of coherent systems over $G(\alpha;n,d,n+m)\times X$. Denote by $p_G$ the projection of $G(\alpha;n,d,n+m)\times X$ onto the first factor. According to \cite[Definition A.6]{bgmmn}, we have a pair $(\mathcal{E},\mathcal{V})$, where $\mathcal{E}$ is a vector bundle on $G(\alpha;n,d,n+m)\times X$ and $\mathcal{V}$ is a locally free subsheaf of $p_{G*}\mathcal{E}$ such that the induced homomorphism $p_G^*\mathcal{V}\to \mathcal{E}$ restricts over $\{(E,V)\}\times X$ to the evaluation map of $(E,V)$. By semicontinuity, the set $S$ of points of $G(\alpha;n,d,n+m)\times X$  at which $p_G^*\mathcal{V}\to \mathcal{E}$ is surjective is open. Now
$$S(\alpha;n,d,n+m)=G(\alpha;n,d,n+m)\setminus p_G((G(\alpha;n,d,n+m)\times X)\setminus S).$$
Since $p_G$ is proper, $p_G((G(\alpha;n,d,n+m)\times X)\setminus S)$ is closed in $G(\alpha;n,d,n+m)$, so $S(\alpha;n,d,n+m)$ is open.

It follows from the definition that $U_g(n,d,n+m)\subset S(\alpha;n,d,n+m)$ for all $\alpha>0$, and from the openness of $\alpha$-stability and the fact that there are only finitely many critical values for $\alpha$ that it is open.

Finally, over $S_0(n,d,n+m)\times X$, the homomorphism $p_X^*\mathcal{V}\to \mathcal{E}$ is surjective, so we have an exact sequence of vector bundles
$$0\lra\mathcal{K}\lra p_X^*\mathcal{V}\lra \mathcal{E}\lra0.$$
Dualising, we obtain a family of coherent systems $(\mathcal{K}^*,\mathcal{V}^*)$, which restricts over $\{(E,V)\}\times X$ to $D(E,V)$. It follows from the openness of $\alpha$-stability that $T(n,d,n+m)$ is open in $S_0(n,d,n+m)$.

If $\gcd(n,d,n+m)>1$, there is no universal family, but universal families exist locally in the \'etale topology. This is sufficient for the proof to go through.
\end{proof}

 For the next result we first recall the definition and properties of the Picard sheaf. First suppose that $\gcd(n,d)=1$ and let $\mathcal{U}$ be the universal  bundle over $X\times M(n,d)$.  The {\it Picard sheaf} $\mathcal{W}$ over $M(n,d)$ is the direct image sheaf $\mathcal{W}:=p_{2*}(\mathcal{U})$ where $p_2$ is the projection on the second factor. For $d\geq n(g-1)$, the open set
$$Z:=\{E\in M(n,d)| h^0(E)=d+n(1-g)\}$$
is non-empty and $\mathcal{W}|_{Z}$ is a vector bundle of rank $d+n(1-g)$.  For $k\le d+n(1-g)$, the Grassmannian bundle $\operatorname{Grass}(k,\mathcal{W}|_Z)$ is smooth and irreducible and
\begin{equation}\label{eqgrass}
\dim \operatorname{Grass}(k,\mathcal{W}|_Z)=n^2(g-1)+1+k(d+n(1-g)-k)=\beta(n,d,k).
\end{equation}
If $\gcd(n,d)\ne1$, the universal bundle does not exist, but one can define a projective Picard bundle over $Z$ (see \cite{bis}) and the associated Grassmannian bundles $\operatorname{Grass}(k,\mathcal{W}|_Z)$, and \eqref{eqgrass} still holds. If $d>2n(g-1)$, then $Z=M(n,d)$ and $\mathcal{W}|_Z=\mathcal{W}$.

\begin{theorem}\label{props0}
Let $X$ be a general curve of genus $g\ge2$ and suppose that $d\ge ng+m$ with $m\ge1$. Then
\begin{enumerate}
\item[(i)] there exists $E\in M(n,d)$ with $E$ generated and $h^0(E)=d+n(1-g)$;
\item[(ii)] $S_0(n,d,n+m)$ possesses an irreducible component which has dimension $\beta(n,d,n+m)$ and is birational to $\operatorname{Grass}(n+m,\mathcal{W}|_Z)$;
\item[(iii)] if $d>2n(g-1)$, $S_0(n,d,n+m)$ is irreducible of dimension $\beta(n,d,n+m)$ and is birational to $\operatorname{Grass}(n+m,\mathcal{W})$;
\end{enumerate}
\end{theorem}

\begin{proof}
(i) and (ii): The open set $Y$ in $G_0(n,d,n+m)$ defined by
$$Y:=\{(E,V)\in G_0(n,d,n+m)|E \mbox{ stable }, h^0(E)=d+n(1-g)\}$$
can be identified with  $\operatorname{Grass}(n+m,\mathcal{W}|_Z)$. In view of \eqref{eqgrass}, this proves the analogue of (ii) for $G_0(n,d,n+m)$. If (i) holds, then $Y\cap S_0(n,d,n+m)\ne\emptyset$ and (ii) follows from the fact that $S_0(n,d,n+m)$ is open in $G_0(n,d,n+m)$ by Lemma \ref{genopen}. It remains only to prove (i), which we shall do by induction on $d\ge ng+1$.

Suppose first that $d=ng+1$ and $m=1$. Then (i) follows from \cite[Theorem 1.1 and formulae (1.4) and (1.6)]{br}. Now suppose that $d>ng+1$ and that $F\in M(n,d-1)$ with $F$ generated and $h^0(F)=d-1+n(1-g)$. Now consider elementary transformations
$$0\lra F\lra E\lra \mathbb{C}_p\lra 0.$$
Since $h^1(F)=0$, the projection $H^0(E)\to \mathbb{C}_p$, induced by the above exact sequence, is surjective. Hence $E$ is generated and $h^0(E)=d+n(1-g)$. Such $E$ is not necessarily stable, but can be deformed to a stable bundle with the same properties. This completes the proof of (i) by induction.

(iii): If $d>2n(g-1)$, then every semistable bundle $E$ of rank $n$ and degree $d$ has $h^0(E)=d+n(1-g)$. It follows that
$$S_0(n,d,n+m)\setminus Y=\{(E,V)\in S_0(n,d,n+m)|E\mbox{ strictly semistable}\}.$$
Recall now that the isomorphism classes of strictly semistable bundles of given rank and degree depend on fewer than $\dim M(n,d)$ parameters (see \cite[Lemma 4.1]{bgn}); it follows that $\dim(S_0(n,d,n+m)\setminus Y)<\beta(n,d,n+m)$, so that $Y\cap S_0(n,d,n+m)$ is dense in $S_0(n,d,n+m)$. This completes the proof of (iii).
\end{proof}

\begin{corollary}\label{cor3.4}
Let $X$ be a general curve of genus $g\ge3$ and $m\ge ng$. Then $S_0(n,d,n+m)\ne\emptyset$ if and only if $d\ge ng+m$.
\end{corollary}
\begin{proof} This follows at once from Theorems \ref{beta}(i) and \ref{props0}.
\end{proof}

\begin{remark}\label{r42}\begin{em}
If $m<ng$, there can be a big difference between the bounds given by Theorems \ref{beta}(i) and \ref{props0}. In particular, if $2m\le d<ng+m$, these theorems do not determine the non-emptiness of $S_0(n,d,n+m)$.
\end{em}\end{remark}

\begin{remark}\label{r43a}\begin{em}
The proof of Theorem \ref{props0} shows that, if $d\ge ng+1$, the general element $E\in M(n,d)$ is generated. It seems likely that this has been known for a long time, but we have been unable to locate a proof in the literature.
\end{em}\end{remark}

Using this theorem, we give an example which shows that Butler's Conjecture can fail when $g=2$.

\begin{example}\label{exg2}\begin{em}
Let $g=2$, $n\ge2$, $3n+2a<d<4n+2a$ for some positive integer $a$ and $m=d-2n-a$. Then the hypotheses of Theorem \ref{props0} are satisfied, so $S_0(n,d,n+m)\ne\emptyset$. Now let $(E,V)\in S_0(n,d,n+m)$. If $D(E,V)\in S_0(m,d,n+m)$, then $M_{V,E}^*$ is semistable of slope $\mu(M_{V,E}^*)=\frac{d}{d-2n-a}>2$ since $d<4n+2a$. Hence $h^1(M_{V,E}^*)=h^0(M_{V,E}\otimes K_X)=0$ and the homomorphism $H^1(E^*)\to V^*\otimes H^1({\mathcal O}_X)$ induced by the dual of \eqref{eq0} is surjective. Hence
$$d+n=h^1(E^*)\ge 2(n+m)=2d-2n-2a,$$
contradicting the assumption that $d>3n+2a$. It follows that $M_{V,E}^*$ is not semistable, so $D(E,V)\not\in S_0(m,d,n+m)$. Thus Butler's Conjecture fails for $(n,d,n+m)$.
\end{em}\end{example}

In general we have the following results.

\begin{proposition}\label{propt}
The formula $(E,V)\mapsto D(E,V)$ defines an isomorphism of schemes
$$D:T(n,d,n+m)\lra T(m,d,n+m).$$
\end{proposition}

\begin{proof}
The fact that $T(n,d,n+m)$ has a natural structure as a scheme follows from Lemma \ref{genopen}. Since $D(D(E,V))=(E,V)$, the result then follows immediately from the definition.
\end{proof}

\begin{theorem}\label{cort}
Butler's Conjecture holds for $(n,d,n+m)$ if and only if $T(n,d,n+m)$ is dense in $S_0(n,d,n+m)$ and $T(m,d,n+m)$ is dense in $S_0(m,d,n+m)$. In particular, if both $S_0(n,d,n+m)$ and $S_0(m,d,n+m)$ are irreducible and $T(n,d,n+m)$ is non-empty, then Butler's Conjecture holds non-trivially for $(n,d,n+m)$.
\end{theorem}

\begin{proof}
This follows at once from Proposition \ref{propt} and Lemma \ref{genopen}.
\end{proof}

\begin{remark}\label{rembalr}\begin{em}
According to \cite[Theorem 2.13]{balr}, $T(n,d,n+m)\ne\emptyset$ when $d$ is sufficiently large compared with $g$, $n$ and $m$.
\end{em}\end{remark}

We can now complete and extend Mercat's result \cite[Ch 2 Th\'{e}or\`{e}me B-1]{mer}.

\begin{theorem}\label{mercat}
Let $X$ be a curve of genus $g\ge2$ and let $m=d-ng$.
\begin{enumerate}
\item[(i)] If $d>2ng-n$, then
\begin{equation}\label{eqmer2}
S_0(n,d,n+m)=G_0(n,d,n+m)=M(n,d)
\end{equation}
and, in particular, $S_0(n,d,n+m)$ is smooth and irreducible of dimension $\beta(n,d,n+m)$.
\item[(ii)] If $d>2ng$, then $S_0(m,d,n+m)=G_0(m,d,n+m)=B(m,d,n+m)$; moreover, Butler's Conjecture holds non-trivially and the dual span construction defines an isomorphism
\begin{equation}\label{mer1}
D:S_0(n,d,n+m)\lra S_0(m,d,n+m).
\end{equation}
\item[(iii)] If $g\ge3$ and $X$ is not hyperelliptic, then $S_0(ng, 2ng,n(g+1))$ is irreducible and
\begin{equation}\label{mer3}
\emptyset\ne B(ng,2ng,n(g+1))\subset T(ng,2ng,n(g+1))\subset S_0(ng,2ng,n(g+1));
\end{equation}
moreover, Butler's Conjecture holds non-trivially for $(n,2ng,n(g+1))$.
\end{enumerate}
\end{theorem}

\begin{proof}
(i) For $d>2ng-n$, any semistable bundle $E$ of rank $n$ and degree $d$ is generated and $h^0(E)=n+m$ by Riemann-Roch. Moreover, if $E$ is strictly semistable and $E'$ is any subbundle of rank $n'$ and degree $d'$ contradicting stability, then $E'$ and $E/E'$ are also semistable, so $h^0(E')= d'+n'(1-g)$ and $h^0(E/E')=(d-d')+(n-n')(1-g)$. Hence $(E',H^0(E'))$ is a subsystem of $(E,H^0(E))$ contradicting $\alpha$-stability for all $\alpha$. This proves \eqref{eqmer2}.

(ii) For $d>2ng$, we have $1<\frac{d}m<2$, so by \cite{mer}, every semistable bundle $F$ of rank $m$ and degree $d$ has
$$h^0(F)\le m+\frac1g(d-m)=n+m.$$
Moreover, if $h^0(F)=n+m$ and $F$ is strictly semistable, the same argument as above shows that $(F,H^0(F))$ is strictly $\alpha$-semistable for all $\alpha>0$. So $G_0(m,d,n+m)=B(m,d,n+m)$.
The isomorphism \eqref{mer1} now follows from \cite[Ch 2 Th\'eor\`eme B-1]{mer} and implies that
$$S_0(m,d,n+m)=G_0(m,d,n+m).$$

(iii) By \cite[Theorem 5.4]{bgmmn2}, $B(ng,2ng,n(g+1))\ne\emptyset$ and $h^0(E)=n(g+1)$ for all $(E,V)\in G_0(ng,2ng,n(g+1))$, so we can regard $B(ng,2ng,n(g+1))$ as a non-empty open subset of $G_0(ng,2ng,n(g+1))$. By \cite[Proposition 2]{mer2}, we have $B(ng,2ng,n(g+1))\subset T(ng,2ng,n(g+1))$. Since $S_0(n,2ng,n(g+1))$ is irreducible by (i) and $S_0(ng,2ng,n(g+1))$ is irreducible by \cite[Theorem 4.4]{bgmmn2}, the result follows from Theorem \ref{cort}.
\end{proof}

\begin{remark}\label{merrem}\begin{em}
For $d>2ng$, the isomorphism \eqref{mer1} holds for every genus. If $g=0$, the spaces concerned are all empty. If $g=1$, they are non-empty if $\gcd(n,d)=1$ and empty otherwise \cite{lnell}. For $g\ge2$, the spaces are non-empty and irreducible.
\end{em}\end{remark}

\section{$T(n,d,n+m)$ and $\mathcal{P}_0(n,d,n+m)$}\label{secug}

In this section, we give some important properties of $T(n,d,n+m)$, $U_g(n,d,n+m)$ and $\mathcal{P}_0(n,d,n+m)$. In particular, those of $\mathcal{P}_0(n,d,n+m)$ will be used in the following sections.

The next theorem demonstrates the relationship between $T(n,d,n+m)$ and $U_g(n,d,n+m)$.

\begin{theorem}\label{main}
$T(n,d,n+m)\subset U_g(n,d,n+m)$.
\end{theorem}

\begin{proof}
Suppose that $(E,V)\in T(n,d,n+m)$. If $(E,V)\not\in U_g(n,d,n+m)$, then (see Remark \ref{subnoalfa}) it possesses a proper coherent subsystem  $(F,W)$ of type $(n',d',k')$ such that
\begin{equation}\label{eq02}
\frac{k'}{n'} >1+\frac{m}{n}.
 \end{equation}

We can certainly assume that $(F,W)$ is  generically generated. The kernel $N$ of the evaluation map $W\otimes \mathcal{O}_X\to F$ then has rank $k'-n'$ and degree $-(d'-\tau)$ for some $\tau\ge0$. Note that $N$ is a subbundle of $M_{{V,E}}$.

 Since $E$ is semistable, we have $$\frac{d'-\tau}{n'}\leq \frac{d'}{n'}\leq \frac{d}{n}.$$
  From this inequality and \eqref{eq02}, we have

 \begin{equation}\label{eq03}
 \mu (N^*) =\frac{d'-\tau}{k'-n'}\leq \frac{n'd}{n(k'-n')} < \frac{n'd}{n'm}= \mu (M^*_{{V,E}}).
 \end{equation}

We conclude from (\ref{eq03}) that $\mu (M_{{V,E}}) < \mu(N)$, hence that $M_{V,E}$ is not semistable, so $D(E,V)\notin S_{0}(m,d,n+m)$, contradicting the hypothesis that $(E,V)\in T(n,d,n+m)$. Therefore $(E,V)\in U_g(n,d,n+m)$ as claimed.
\end{proof}

\begin{remark}\label{remdim1}\begin{em}
This theorem suggests a possible method of showing that $T(n,d,n+m)\ne\emptyset$ using the large-$\alpha$ moduli space $G_L(n,d,n+m)$, at least when $\gcd(n,m)=1$. In this case, by \cite[Proposition 5.9]{bgmn}, the dual span defines an injective morphism
\begin{equation}\label{equg}
D: U_g(n,d,n+m) \to G_L(m,d,n+m).
\end{equation}
Suppose further that $U_g(n,d,n+m)\ne\emptyset$, $\dim G_L(n,d,n+m)=\beta(n,d,n+m)$ and  $U_g(m,d,n+m)$ is dense in $G_L(m,d,n+m)$. Then the image of the morphism \eqref{equg} contains an open set in $G_L(m,d,n+m)$ which meets $U_g(m,d,n+m)$ and it follows that $T(n,d,n+m)\ne\emptyset$. We propose to investigate this possibility in the future.
\end{em}\end{remark}

\begin{proposition}\label{dimcomp} Suppose $d>(\frac{2nm}{n+m})(g-1)$ and $T(n,d,n+m)$ is non-empty. Then $T(n,d,n+m)$ is smooth of the expected dimension $\beta(n,d,n+m)$. Hence $U_g(n,d,n+m)$ and $U_g(m,d,n+m)$ have an irreducible component of dimension $ \beta(n,d,n+m)=\beta(m,d,n+m)$.
\end{proposition}

\begin{proof} Let $(E,V)\in T(n,d,n+m)$. From the sequence \eqref{eq0}, it follows that the kernel of the Petri map \eqref{eqpetri} is $H^0(M_{V,E}\otimes E^*\otimes K_X)$. Since
$$\mu (M_{V,E}\otimes E^*\otimes K_X)=-\frac{d}m-\frac{d}n+2g-2<0$$
and $E$ and $M_{V,E}$ are both semistable, it follows that $H^0(M_{V,E}\otimes E^*\otimes K_X)=0$.
So the Petri map is injective and $T(n,d,n+m)$ is smooth of dimension $\beta(n,d,n+m)$ as claimed. The same holds for $T(m,d,n+m)$ by Proposition \ref{propt}. Since $T(n,d,n+m)$ is open in $U_g(n,d,n+m)$ by Lemma \ref{genopen} and Theorem \ref{main}, the last part of the proposition follows.
\end{proof}



In order to determine the truth or falsity of Butler's Conjecture, it is important to investigate cases in which $(E,V)\in S_0(n,d,n+m)$, but $D(E,V)\notin S_0(m,d,n+m)$.

\begin{lemma}\label{lemma1}
Let $(E,V)\in S_0(n,d,n+m)$. If $(E,V)$ has a  subsystem $(G,W)$ of type $(s,d',s+t)$ such that
$\frac{s}{t}\le \frac{n}{m}$, then $D(E,V)\notin S_0(m,d,n+m)$.
\end{lemma}

\begin{proof} We can certainly assume that $(G,W)\subset (E,V)$ is a generically generated subsystem of type $(s,d',s+t)$. The kernel of the evaluation map $W\otimes \mathcal{O}_X\to G$ defines a subbundle $M$ of $ M_{V,E}$ of slope $\mu (M) =-\frac{d'-\tau }{t}$.
Suppose, contrary to our claim, that $(M_{V,E}^*,V^*)\in S_0(m,d,n+m)$. As $ M_{V,E}^*$ is semistable, we have $$\frac{d}{m}=\mu ( M_{V,E}^*)\leq \mu (M^*)=\frac{d'-\tau}{t}\leq \frac{d'}{t}.$$ Hence $\frac{d}{d'}\leq \frac{m}{t}$, but by hypothesis
$$\frac{d}{d'}\leq
\frac{m}{t}\le\frac{n}{s}.
$$
Thus $\mu(E)=\frac{d}{n}\le\frac{d'}{s}=\mu (G)$.
This contradicts the semistability of $E$ unless all the inequalities are equalities. But, in the latter case, $\frac{s}t=\frac{n}m$ and $(M^*,W^*)$ is a quotient coherent system of $D(E,V)$ of type $(t,d',s+t)$ with $\frac{d'}t=\frac{d}m$, which contradicts the $\alpha$-stability of $D(E,V)$ for all $\alpha>0$.
\end{proof}

By a {\it subpencil} of a coherent system $(E,V)$, we mean a subsystem $(L,W)$ of type $(1,d',2)$. If $(E,V)$ admits a subpencil, we say also that $E$ admits a subpencil. Denote by $\mathcal{P}(\alpha;n,d,n+m)\subset G(\alpha;n,d,n+m)$ the locus of generated coherent systems in $ G(\alpha;n,d,n+m)$ that admit subpencils. As usual, we write also $\mathcal{P}_0(n,d,n+m)$ when $\alpha>0$ is small.

\begin{lemma}\label{pencil} $\mathcal{P}(\alpha;n,d,n+m)$ is closed in $S(\alpha;n,d,n+m)$.
\end{lemma}

\begin{proof} Working
locally in the \'etale topology if necessary, we can assume without loss of generality that there is a family $(\mathcal{E}, \mathcal{V})$ of coherent systems parameterized by $S(\alpha;n,d,n+m)$. Let $(\mathcal{L}, \mathcal{W})$ be the family of linear systems of type $(e,2)$ parameterized by $G(1,e,2)$. Take the pull-back of the families $(\mathcal{E}, \mathcal{V})$ and $(\mathcal{L}, \mathcal{W})$ to the product
$$G(1,e,2) \times S(\alpha;n,d,n+m)\times X.$$
We conclude, from the semicontinuity of $\operatorname{Hom}$, that the support  $Z_e\subset G(1,e,2) \times S(\alpha;n,d,n+m)$ of  $p_{12*}(\it{Hom}(p_{13}^*(\mathcal{L}, \mathcal{W}),p_{23}^*(\mathcal{E}, \mathcal{V}))) $
is closed, hence that $p _{2}(Z_e)\subset S(\alpha;n,d,n+m) $ is closed, since the projection $p _2:G(1,e,2) \times S(\alpha;n,d,n+m)\to  S(\alpha;n,d,n+m) $ is proper. Since there are finitely many possible choices for $e$ for which $Z_e\ne\emptyset$,
we conclude that $\mathcal{P}(\alpha;n,d,n+m)= \bigcup _e  p _{2}(Z_e)$ is closed.
\end{proof}

\begin{proposition}\label{segrerk2} Let $X$ be a general curve. If $d\le n\left\lfloor\frac{g+3}2\right\rfloor$ and $m\le n$, then
$\mathcal{P}_0(n,d,n+m) = \emptyset $.
\end{proposition}

\begin{proof} Let $(E,V)\in S_0(n,d,n+m)$ and let $(L,W)$ be a subsystem of type $(1,d',s)$. By classical Brill-Noether theory, $\left\lfloor\frac{g+3}2\right\rfloor\leq d'$ if
$s\geq 2 $. But, from the semistability of $E$ and the hypothesis,
$$\left\lfloor\frac{g+3}2\right\rfloor\leq d'\leq \frac{d}{n} \le\left\lfloor\frac{g+3}2\right\rfloor.$$
This is possible only if all the inequalities are equalities. So $(L,W)\in G(1,\frac{d}n,2)$, which contradicts the $\alpha$-stability of $(E,V)$ for all $\alpha>0$. Therefore $s\leq 1$, which proves the proposition.
\end{proof}

\begin{proposition}\label{coropencil}If $m\le n$, then
$${\mathcal P}_0(n,d,n+m)\cap T(n,d,n+m)=\emptyset.$$
 \end{proposition}

 \begin{proof} This follows from Lemma \ref{lemma1}.
\end{proof}

Proposition \ref{coropencil} implies that, if $m\le n$ and $\mathcal{P}_0(n,d,n+m)$ contains a complete irreducible component of $S_0(n,d,n+m)$, then Butler's Conjecture must fail. On the other hand, if $\mathcal{P}_0(n,d,n+m)\ne\emptyset$ and $\dim\mathcal{P}_0(n,d,n+m)<\beta(n,d,n+m)$, then $S_0(n,d,n+m)\ne\emptyset$ and $\mathcal{P}_0(n,d,n+m)$ contains no irreducible component of $S_0(n,d,n+m)$. We shall find examples of this in Sections \ref{sec2d4} and \ref{secg6} which allow us to prove Butler's Conjecture in some cases.



\section{Coherent systems of type $(2,d,4)$}\label{sec2d4}

In this section, we consider existence problems for coherent systems of type $(2,d,4)$ and make some deductions concerning Butler's Conjecture. One may note that this is the first unknown case for the conjecture.


We can certainly assume that $d\geq 2+\frac{2g}{3}$,  since otherwise, by  Theorem \ref{beta}, Butler's conjecture is trivially fulfilled. However, as we shall see, we have substantially stronger results in this case.

\begin{proposition}\label{proposition2} Let $X$ be any curve. Then
\begin{equation}\label{eqequal}
U_g(2,d,4)=S_0(2,d,4).
\end{equation}
\end{proposition}

\begin{proof} Certainly $U_g(2,d,4)\subset S_0(2,d,4)$. For the reverse inclusion, let $(E,V)\in S_0(2,d,4)$ and let $(L,W)$ be a subsystem of $(E,V)$ of type $(1,d',k')$. Since $E$ is semistable, $d'\le \frac{d}2$ and, by Remark \ref{rem1g}, $k'\leq 2$. If $d'=\frac{d}2$ and $k'=2$, then $(L,W)$ contradicts the assumption that $(E,V)\in S_0(2,d,4)$. Hence, either $d'<\frac{d}2$ or $k'<2$ and
$$\mu _{\alpha} (L,W)<\mu _{\alpha} (E,V)$$
for all $\alpha >0$. This proves that $S_0(2,d,4)=U_g(2,d,4)$.
\end{proof}

\begin{theorem}\label{p3.1} Let $X$ be a general curve of genus $g\ge3$.
 \begin{enumerate}
\item[(i)] If $d\le g+2$ and $g$ is odd, or $d\le g+1$ and $g$ is even, then $G_0(2,d,4)=\emptyset$.
\item[(ii)] If $d\ge g+4$ or $g\ge4$, $d\ge g+3$, then $G_0(2,d,4)\ne\emptyset$.
\item[(iii)] If $d= g+3$ and $g$ is odd, or $d=g+2$ and $g$ is even, then $\mathcal{P}_0(2,d,4)=\emptyset$ and $S_0(2,d,4)=G_0(2,d,4)$; moreover, if $(E,V)\in G_0(2,d,4)$, then $E$ is stable.
\item[(iv)] If $d\ge 2g+2$, then $S_0(2,d,4)\ne\emptyset$.
\item[(v)] If $d>4g-4$, then $S_0(2,d,4)$ is irreducible and birational to $\Grass(4,\mathcal{W})$.
\end{enumerate}
\end{theorem}

\begin{proof}
(i) For $g\ge4$, this follows from the fact that $\Cliff_2(X)=\Cliff(X)=\left\lfloor\frac{g-1}2\right\rfloor$  \cite{bf}. For $g=3$, $\Cliff_2(X)$ is not defined. However, in this case, if $E$ is semistable of rank $2$ with $d_E\le4$, then $h^0(E)\le3$ by \cite[Theorem 4.15(a)]{lnci}. If $d_E=5$, then $h^0(K_X\otimes E^*)\le2$ by \cite[Corollary 4.12]{lnci}; hence, by Riemann-Roch, $h^0(E)\le3$.

(ii) Except when $g$ is odd and $d=g+3$, this follows from \cite{T2}. For $g$ odd, $d=g+3$, $g\ge9$, see \cite[Theorem 1.1]{fo} (this theorem is stated in \cite{fo} for arbitrary odd genus, but is proved only for $g\ge9$). For $g=5$, we have $B(2,8,4)\ne\emptyset$ by \cite[section 3]{berf}. For $g=7$, see \cite[Postscript]{gmn}.


(iii) The fact that $\mathcal{P}_0(2,d,4)=\emptyset$ is a special case of Proposition \ref{segrerk2}.

Now suppose $(E,V)\in G_0(2,d,4)$. Since $(E,V)$ does not admit a subpencil, any quotient  $Q$ of $E$ must have $h^0(Q)\ge3$. By classical Brill-Noether theory, this implies that $d_Q\ge \left\lceil\frac{2g}3\right\rceil+2>\frac{d}2$; so $E$ is stable. If $(E,V)$ is not generated, then $G_0(2,d-1,4)\ne\emptyset$, contradicting (i). This completes the proof.

(iv) $S_0(2,d,4)\ne\emptyset$ for $d\ge 2g+2$ by Theorem \ref{props0}.

(v) follows from Theorem \ref{props0}.
\end{proof}

\begin{remark}\label{rem3.1}\begin{em}\
\begin{enumerate}
\item If $g=3$, $G_0(2,6,4)=\emptyset$. In fact, if $g=3$ and $(E,V)\in G_0(2,6,4)$, $E$ cannot be stable since  $B(2,6,4)=\emptyset$ (see, for example, \cite[section 5]{gmn}), while, if $E$ is strictly semistable, it must admit a subpencil, contradicting $\alpha$-stability for all $\alpha>0$. So $G_0(2,6,4)=\emptyset$.
\item For $g=4$, $G_0(2,7,4)\ne\emptyset$, but $S_0(2,7,4)=\emptyset$.  In fact, in this case, $h^0(\det E)<5$ and it follows from results of \cite{ln4} that any $E\in B(2,7,4)$ has the form $E\simeq K_X\otimes E_L^*$ (with the notation of \cite{ln4}). But then (see \cite[Lemma 5.9]{ln4}), $E$ possesses a quotient bundle of the form $T(p)$, where $T$ is either of the trigonal bundles on $X$ and $p\in X$. Since $T(p)$ is not generated, it follows that $E$ is not generated. Hence $S_0(2,7,4)=\emptyset$ when $g=4$.
\item For $g=5$,  $S_0(2,8,4)=G_0(2,8,4)$ is irreducible of dimension $2$ (see \cite[Propositions 5.1 and 5.3]{news1}). Here $\beta(2,8,4)=1$.
\item For $g=7$, the same is true for $S_0(2,10,4)=G_0(2,10,4)$ (see \cite[Postscript]{gmn}). Again $\beta(2,10,4)=1$.
\item For $g$ odd, $g\ge9$, $S_0(2,g+3,4)=G_0(2,g+3,4)$ has an irreducible component of dimension $2$ by \cite[Theorem 1.1]{fo}.
\end{enumerate}
\end{em}\end{remark}

\begin{remark}\label{r43}\begin{em}
For $g$ even, if $(E,V)\in G_0(2,g+2,4)$ then $(E,V)$ does not admit a subpencil, so $h^0(\det E)\ge5$ by \cite[Lemma 3.9]{pr}. This is impossible for $4\le g\le 8$, so in these cases $G_0(2,g+2,4)=\emptyset$. We know also that, when $g=10$, $G_0(2,12,4)=\emptyset$ (see \cite[Theorem 4.1]{gmn}). It seems reasonable to conjecture that $G_0(2,g+2,4)=\emptyset$ for all $g$; certainly $\beta(2,g+2,4)=-3<0$.
\end{em}\end{remark}

Recall now that, by Proposition \ref{propt}, the morphism $D:T(2,d,4)\to T(2,d,4)$ is always an automorphism. The next proposition asserts that, for sufficiently small $d$, $T(2,d,4)$ is as large as is permitted by Proposition \ref{coropencil}.

\begin{proposition} \label{but1}  Let $X$ be a general curve and  $d \leq 2\left(\left\lceil\frac{2g}{3}\right\rceil+2\right) $. Then $$T(2,d,4)=S_0(2,d,4)\setminus\mathcal{P}_0(2,d,4).$$
\end{proposition}

\begin{proof} Note first that $T(2,d,4)\subset S_0(2,d,4)\setminus\mathcal{P}_0(2,d,4)$ by Proposition \ref{coropencil}. Conversely, let $(E,V)\in S_0(2,d,4)\setminus\mathcal{P}_0(2,d,4)$ and let $L$ be a quotient line bundle of $M_{V,E}^* $. Let $W$ denote the image of $V^*\subset H^0(M_{V,E}^*)$ in $H^0(L)$. Then $W$ generates $L$ and $L\not\simeq {\mathcal O}_X$ since $H^0(M_{V,E})=0$. So we have a diagram

\begin{equation}\label{diagr2}
\begin{array}{ccccccccc}
0&\lra& E^*&\lra &
V^*\otimes \mathcal{O}_X &\lra &M_{V,E}^*&\lra &0\\
&& \downarrow&&\downarrow&&\downarrow&\\
 0&\lra& F^*&\lra &
W\otimes \mathcal{O}_X &\lra &L&\lra &0\\
&& &&\downarrow&&\downarrow&\\
&&&&0&&0\\
\end{array}
\end{equation}
with $\dim W\ge2$. The homomorphism $E^*\to F^*$ is not necessarily surjective, but it cannot be $0$.

If $\dim W=2$, then $F$ is a line bundle. Dualising \eqref{diagr2}, we obtain a non-zero homomorphism $F\to E$ and hence a subpencil $(F',W^*)$ of $(E,V)$, where $F'$ is the saturation of the image of $F$ in $E$. This contradicts our assumption. Hence $\dim W\ge3$ and, by classical Brill-Noether theory,
$$d_L\ge\left\lceil\frac{2g}{3}\right\rceil+2\ge\frac{d}2=\mu(M_{V,E}^*).$$
Hence $M_{V,E}^*$ is semistable and we have proved also that $D(E,V):=(M_{V,E}^*,V^*)$ does not admit a subpencil. So $D(E,V)\in S_0(2,d,4)\setminus\mathcal{P}_0(2,d,4)$. Hence $(E,V)\in T(2,d,4)$.
\end{proof}

\begin{corollary}\label{c45}
Let $X$ be a general curve of odd genus $g\ge5$. Then Butler's Conjecture holds non-trivially for $(2,g+3,4)$.
\end{corollary}

\begin{proof}
This follows from Theorem \ref{p3.1}(ii) and (iii) and Proposition \ref{but1}.
\end{proof}

In order to use Proposition \ref{but1} to obtain further cases in which Butler's Conjecture holds non-trivially for $(2,d,4)$, we need to study $\mathcal{P}_0(2,d,4)$. In fact, if $(E,V)\in \mathcal{P}_0(2,d,4)$, we have a non-split exact sequence
\begin{equation}\label{eqp1}
0\lra (L_1,V_1)\lra (E,V)\lra (L_2,V_2)\lra0,
\end{equation}
with  $\dim V_1\ge2$. Since $V_2$ generates $L_2$ and $h^0(E^*)=0$, we must have $\dim V_2\ge2$, hence $\dim V_1=\dim V_2=2$. For fixed $(L_1,V_1)\in G(1,d_1,2)$, $(L_2,V_2)\in G(1,d_2,2)$ with $d_2=d-d_1>d_1$, there exists a non-split extension \eqref{eqp1} if and only if
$$\dim\Ext^1((L_2,V_2),(L_1,V_1))\ge1.$$
 Now $\Hom((L_2,V_2),(L_1,V_1))=0$ since $d_2>d_1$, so, by \cite[(8),(9) and (11)]{bgmn}
\begin{equation}\label{eqext}
\dim\Ext^1((L_2,V_2),(L_1,V_1))=C_{21}+h^0(L_1^*\otimes L_2^*\otimes K_X),
\end{equation}
where
\begin{equation}\label{eqc21}
C_{21}=g-1-d_1+d_2+2d_1-2(g-1)-4=d-g-3.
\end{equation}

\begin{proposition}\label{p47}
Let $X$ be a general curve. If
\begin{equation}\label{eq43}
g+4\le d\le2\left(\left\lceil\frac{2g}{3}\right\rceil+2\right),
\end{equation}
then $\mathcal{P}_0(2,d,4)\ne \emptyset$.
\end{proposition}

\begin{proof}
Suppose that $d$ satisfies \eqref{eq43}. Then, by \eqref{eqext} and \eqref{eqc21}, there exists a non-split extension \eqref{eqp1} with $d_1=\left\lfloor\frac{g+3}2\right\rfloor$ and $d_2=d-d_1>d_1$. By classical Brill-Noether theory, we can assume that $(L_1,V_1)$ and $(L_2,V_2)$ are both generated. If $E$ is stable, then certainly $(E,V)\in\mathcal{P}_0(2,d,4)$. If $E$ is not semistable, then $E$ admits a line subbundle $L$ of degree $d_L>\frac{d}2$. Then $E/L$ is a line bundle of degree
$$d_{E/L}<\frac{d}2\le\left\lceil\frac{2g}{3}\right\rceil+2.$$
It follows by classical Brill-Noether theory that $h^0(E/L)\le2$, so $W=H^0(L)\cap V$ has dimension $\ge2$. Now \eqref{eqp1} induces a non-zero homomorphism $(L,W)\to (L_2, V_2)$, which is necessarily injective; moreover $\dim W=\dim V_2=2$. If $d_L<d_2$, this contradicts the fact that $(L_2,V_2)$ is generated. On the other hand, if $d_L=d_2$, then \eqref{eqp1} splits, another contradiction.

Now suppose that $E$ is strictly semistable and let $L$ be a line subbundle of $E$ with $d_L=\frac{d}2$ and $W=H^0(L)\cap V$. If $\dim W=2$, we obtain again a non-zero homomorphism $(L,W)\to (L_2,V_2)$, which contradicts the assumption that $(L_2,V_2)$ is generated. Hence $\dim W=1$ for all such $L$, so that $(E,V)\in\mathcal{P}_0(2,d,4)$.
\end{proof}

\begin{proposition}\label{p3.2}
Let $X$ be a general curve. If there exists an integer $r\ge1$ such that
\begin{equation}\label{eqp2}
d\ge\max\left\{g-r+\left\lfloor\frac{g}r\right\rfloor,g+r+3\right\},
\end{equation}
then
\begin{equation}\label{eqp0}
\dim\mathcal{P}_0(2,d,4)<\beta(2,d,4).
\end{equation}
Hence $\mathcal{P}_0(2,d,4)$ does not contain any irreducible component of $S_0(2,d,4)$.
\end{proposition}

\begin{proof}
Any $(E,V)\in\mathcal{P}_0(2,d,4)$ can be inserted in a non-split exact sequence \eqref{eqp1}. For fixed $(L_1,V_1)$, $(L_2,V_2)$, such extensions depend on
$$\dim\Ext^1((L_2,V_2),(L_1,V_1))-1$$
parameters.  By \cite[Corollary 3.7]{bgmn} and noting that here $C_{12}=C_{21}$, \eqref{eqp0} would follow if we had
\begin{equation}\label{eqp3}
d-g-3>h^0(L_1^*\otimes L_2^*\otimes K_X)
\end{equation}
for all possible choices of $(L_1,V_1)$, $(L_2,V_2)$. If $r\ge1$ and $d\ge g+r+3$, then \eqref{eqp3} will hold provided that $h^0(L_1^*\otimes L_2^*\otimes K_X)\le r-1$. Since $\deg(L_1^*\otimes L_2^*\otimes K_X)=2g-2-d$, classical Brill-Noether theory will ensure that this is true provided that
$$2g-2-d< g+r-1-\left\lfloor\frac{g}r\right\rfloor.$$
This is equivalent to $d\ge g-r+\left\lfloor\frac{g}r\right\rfloor$.
\end{proof}

\begin{proposition}\label{p510}
Let $X$ be a general curve of genus $g$, $3\le g\le5$. Then
\begin{enumerate}
\item[(i)]  $S_0(2,d,4)\ne\emptyset$ if and only if $d\ge g+4$ or $g=5$, $d=8$;
\item[(ii)] $\mathcal{P}_0(2,d,4)$ does not contain any irreducible component of $S_0(2,d,4)$.
\end{enumerate}
\end{proposition}
\begin{proof}
(i) If $g=3$ or $g=4$, $S_0(2,d,4)=\emptyset$ for $d\le g+3$ by Theorem \ref{p3.1}(i) and Remark \ref{rem3.1}(1), (2) and for $g=5$, $d\le7$ by Theorem \ref{p3.1}(i). The rest follows from Theorem \ref{p3.1}(iv) and Propositions \ref{p47} and \ref{p3.2}.

(ii) For $g\le5$, $r=1$, \eqref{eqp2} becomes $d\ge g+4$. The result follows from Proposition \ref{p3.2} and (for $g=5$, $d=8$) Theorem \ref{p3.1}(iii).
\end{proof}

\begin{theorem}\label{cor3.2}
Let $X$ be a general curve of genus $g\ge3$. Then there exist integers $d$ such that
\begin{equation}\label{eq3.2}
\max\left\{g-r+\left\lfloor\frac{g}r\right\rfloor,g+r+3\right\}\le d\le2\left(\left\lceil\frac{2g}3\right\rceil+2\right)\mbox{ for some integer }r\ge1.
\end{equation}
Moreover, if $d$ satisfies these inequalities, then
\begin{enumerate}
\item[(i)] $T(2,d,4)=S_0(2,d,4)\setminus \mathcal{P}_0(2,d,4)$;
\item[(ii)] $\mathcal{P}_0(2,d,4)\ne\emptyset$ and $\dim\mathcal{P}_0(2,d,4)<\beta(2,d,4)$;
\item[(iii)] Butler's Conjecture holds non-trivially for $(2,d,4)$.
\end{enumerate}
If, in addition, $d\ge2g-1$, then $\dim S_0(2,d,4)=\beta(2,d,4)$ and $T(2,d,4)$ is smooth.
\end{theorem}

\begin{proof}
For $r=1$, we can take $d=\max\{2g-1,g+4\}$ in \eqref{eq3.2} for $g\le8$. Similarly $r=3$ gives values of $d$ for $g\ge4$.

(i) follows from Proposition \ref{but1} and (ii) from Propositions \ref{p47} and \ref{p3.2}. (iii) now follows from (i) and (ii). The last assertion follows from Proposition \ref{dimcomp} and (iii).
\end{proof}

\begin{corollary}\label{cor511}
Let $X$ be a general curve of genus $g$, $3\le g\le5$. If
$$g=3, d=7,8,\quad g=4, 8\le d\le10,\quad g=5, 8\le d\le12,$$
then Butler's Conjecture holds non-trivially for $(2,d,4)$.
\end{corollary}

\begin{proof}
This follows from Theorem \ref{cor3.2} and (for $g=5$, $d=8$) Corollary \ref{c45}.
\end{proof}

In the next section, we will investigate the case $g=6$.

\section{Genus 6}\label{secg6}

Higher rank Brill-Noether theory for genus $6$ is particularly interesting as several new phenomena appear (see \cite{news1,ln6}). In our context, we have the following result.

\begin{theorem}\label{p3.3}
Let $X$ be a general curve of genus $6$.
\begin{enumerate}
\item[(i)] $G_0(2,d,4)\ne\emptyset$ if and only if $d\ge9$.
\item[(ii)] $\dim\mathcal{P}_0(2,9,4)=\beta(2,9,4)=1$, $\dim S_0(2,9,4)\ge2$; moreover, if $(E,V)\in S_0(2,9,4)$, then $\det E\simeq K_X(-p)$ for some $p\in X$ and ${\mathcal P}_0(2,9,4)$ does not contain any component of $S_0(2,9,4)$.
\item[(iii)] Butler's Conjecture holds non-trivially for $(2,9,4)$; moreover, the morphism $D:T(2,9,4)\to T(2,9,4)$
is the identity morphism.
\item[(iv)] $\dim\mathcal{P}_0(2,10,4)<\beta(2,10,4)$ and  Butler's Conjecture holds non-trivially for $(2,10,4)$.
\item[(v)]  If $d=11$ or $d=12$, then $\dim\mathcal{P}_0(2,d,4)<\beta(2,d,4)$, $T(2,d,4)$ is smooth of dimension $\beta(2,d,4)$ and  Butler's Conjecture holds non-trivially for $(2,d,4)$.
\item[(vi)]  If $d\ge13$, then $\dim\mathcal{P}_0(2,d,4)<\beta(2,d,4)$.
\item[(vii)] If $d\ge14$, then $S_0(2,d,4)\ne\emptyset$.
\end{enumerate}
\end{theorem}

\begin{proof} (i) For $d\le7$, $G_0(2,d,4)=\emptyset$ by Theorem \ref{p3.1}(i). For $d=8$, note that, if $(E,V)\in G_0(2,8,4)$, then $E$ computes the Clifford index $\Cliff_2(X)$. Now, by \cite[Proposition 5.10]{lnrk2}, $(E,V)$ cannot be $\alpha$-stable. For $d\ge9$, $G_0(2,d,4)\ne\emptyset$ by Theorem \ref{p3.1}(ii).

(ii) If $(E,V)\in\mathcal{P}_0(2,9,4)$, then we have a non-trivial exact sequence \eqref{eqp1} in which $(L_1,V_1)\in G(1,4,2)$ and $(L_2,V_2)$ is a generated element of $G(1,5,2)$. Now consider \cite[(8)]{bgmn}. We have $C_{21}=d-g-3=0$ and $\Hom((L_2,V_2),(L_1,V_1))=0$; so
$$\dim\Ext^1((L_2,V_2),(L_1,V_1))=h^0(L_1^*\otimes L_2^*\otimes K_X).$$
Now $h^0(L_1^*\otimes L_2^*\otimes K_X)=0$ unless $L_2\simeq L_1^*\otimes K_X(-p)$ for some $p\in X$, in which case $h^0(L_1^*\otimes L_2^*\otimes K_X)=1$. Since $G(1,4,2)$ is finite by classical Brill-Noether theory, this shows that $\dim\mathcal{P}_0(2,9,4)\le1$. Now observe that $(L_2,V_2)$ is generated for all but finitely many $p$. In fact, if $(L_2,V_2)$ is not generated, then $L_2\simeq L(q)$ for some $L\in B(1,4,2)$ and some $q\in X$, and $L_1^*\otimes L^*\otimes K_X\simeq\mathcal{O}_X(p,q)$. Since there exist finitely many choices for $L_1$, $L$ and a unique choice of $\{p,q\}$ for each such choice, this justifies our assertion. This completes the proof that $\dim\mathcal{P}_0(2,9,4)=1$ and shows also that, if $(E,V)\in\mathcal{P}_0(2,9,4)$, then $\det E\simeq K_X(-p)$.

Now suppose $(E,V)\in S_0(2,9,4)\setminus\mathcal{P}_0(2,9,4)$. Then, by \cite[Lemma 3.9]{pr}, $h^0(\det E)\ge5$ and it follows that $\det E\simeq K_X(-p)$ for some $p\in X$. Now consider the morphism
$$S_0(2,9,4)\lra B(1,9,5): E\mapsto\det E.$$
By \cite[Theorem 1.1]{o}, every component of every fibre of this morphism has dimension
$\ge 18-3-4(4-9+10)+6=1$. Since, by the previous argument, all but finitely many of these fibres are non-empty,
it follows that $\dim S_0(2,9,4)\ge2$ as asserted. Moreover, each fibre contains finitely many points of ${\mathcal P}_0(2,9,4)$, so ${\mathcal P}_0(2,9,4)$ does not contain any component of $S_0(2,9,4)$.

(iii) It follows from Proposition \ref{but1} that $T(2,9,4)=S_0(2,9,4)\setminus \mathcal{P}_0(2,9,4)$ and then from (ii) that Butler's Conjecture holds non-trivially.

Now let $(E,V)\in T(2,9,4)$. We have already observed in (ii) that $h^0(\det E)\ge5$; moreover $h^0(\det E)\le5$ by classical Brill-Noether theory. We proceed as in the proof of \cite[Proposition 2.4]{gmn}. We have exact sequences
$$0\lra F\lra\bigwedge^2V\otimes\mathcal{O}_X\lra\det E\lra0$$
and
$$0\lra \bigwedge^2M_{V,E}\lra F\lra M_{V,E}\otimes E\lra0.$$
We deduce first that $h^0(F)>0$ and then that $h^0(M_{V,E}\otimes E)>0$. Since $E$ and $M_{V,E}^*$ are both stable of the same slope, it follows that $M_{V,E}^*\simeq E$. Finally, since $\Cliff_2(X)=2$, $V=H^0(E)$, so $D(E,V)\simeq (E,V)$.

(iv) Proposition \ref{p3.2} does not apply, but we can still use the argument in the proof of this proposition to show that $\dim\mathcal{P}_0(2,10,4)<\beta(2,10,4)$. In fact, for $(E,V)\in {\mathcal P}_0(2,10,4)$, we have a sequence \eqref{eqp1} with $(L_1,V_1)\in G(1,4,2)$ and $(L_2,V_2)\in G(1,6,2)$. In general, $h^0(L_1^*\otimes L_2^*\otimes K_X)=0$ and \eqref{eqp3} holds. The only case in which this does not occur is when $L_2\simeq L_1^*\otimes K_X$; according to \cite[Corollary 3.7]{bgmn}, we therefore need to prove that
$$d-g-3+\beta(1,6,2)>1.$$
This is clear. Now Butler's Conjecture holds by Propositions \ref{but1} and \ref{p47}.

(v) This follows from Theorem \ref{cor3.2}.

(vi) This follows by taking $r=1$ in \eqref{eqp2}.

(vii) follows from Theorem \ref{p3.1}(iv).
\end{proof}

We now consider coherent systems of type $(2,d,n+2)$ with $n\ge3$. The next theorem is a reformulation of some of the results in \cite{news1} and \cite{ln6} in terms of coherent systems.

\begin{proposition}\label{newst26} Let $X$ be a general curve of genus $6$.
\begin{enumerate}
\item[(i)] If $n\ge3$, then $G_0(2,d,n+2)=\emptyset$ and $G_0(n,d,n+2)=\emptyset$ for $d\le9$.
\item[(ii)] If $n\ge4$, then $G_0(2,10,n+2)=\emptyset$ and $G_0(n,10,n+2)=\emptyset$.
\item[(iii)] $G_0 (2,10,5)=S_0(2,10,5)=\mathcal{P}_0(2,10,5)$ and consists of a single point $(E,V)$ with $E$ stable. Moreover, Butler's Conjecture holds non-trivially for $(2,10,5)$.
\item[(iv)] $G_0(2,d,5)\ne\emptyset$ for $d\ge11$; $S_0(2,d,5)\ne\emptyset$ for $d\ge15$.
\item[(v)] $G_0(3,d,5)\ne\emptyset$ for $d\ge11$.
\end{enumerate}
\end{proposition}

\begin{proof} (i) and (ii) The results for $(2,d,n+2)$ follow immediately from the fact that $\Cliff_2(X)=2$. For $(n,d,n+2)$, this doesn't work since the corresponding bundles do not contribute to $\Cliff_n(X)$. However, see \cite[Figure 5]{ln6} for $n\le5$ and \cite{mer} for $n\ge6$.

(iii) This follows from \cite[Theorem 4.1 and Proposition 4.4]{news1}.

(iv) By \cite[Proposition 7.2]{ln6}, we have $B(2,d,5)\ne\emptyset$ for $d\ge11$; hence also $G_0(2,d,5)\ne\emptyset$. If $d\ge15$, then $S_0(2,d,5)\ne\emptyset$ by Theorem \ref{props0}.

(v) This follows from \cite[Proposition 7.4 and Remark 7.5]{ln6} except when $d=12$. For this case, choose three non-isomorphic line bundles $L_1$, $L_2$, $L_3$ in $B(1,4,2)$ and define $E:=L_1\oplus L_2\oplus L_3$. Now $h^0(E)=6$; choose a subspace $V$ of $H^0(E)$ of dimension $5$ such that $\dim(V\cap H^0(L_i))=1$ for all $i$ and $\dim(V\cap H^0(L_i\oplus L_j))=3$ for all $i\ne j$. Then $(E,V)\in G_0(3,12,5)$.
\end{proof}

\begin{remark}\label{rem3.4}\begin{em}
Note that, in this proposition, ${\mathcal P}_0(2,10,5)\ne\emptyset$, but ${\mathcal P}_0(3,10,5)=\emptyset$ (for confirmation of this, see the proof of \cite[Proposition 4.4]{news1}). This is compatible with Proposition \ref{coropencil}.
\end{em}\end{remark}



\begin{thebibliography}{9999}

\bibitem{an}
M. Aprodu and J. Nagel,
\lq Koszul cohomology and algebraic geometry\rq,
Amer. Math. Soc. University Lecture Series, Vol. 52, Providence, RI, 2010.







\bibitem{bf}
B. Bakker and G. Farkas,
\lq The Mercat conjecture for stable rank 2 vector bundles on generic curves\rq,
arXiv 1511.03253, to appear in Amer. J. Math.


\bibitem{ball}
E. Ballico,
\lq Curves in Grassmannians and spanned stable bundles.\rq
 {\em  Math. Nachr.} {\bf 220} (2000), 5--10.

\bibitem{balr}
E. Ballico and L. Ramella,
\lq The restricted tangent bundle of smooth curves in Grassmannians and curves in flag varieties\rq,
{\em Rocky Mountain J. Math.} {\bf 30} (2000), 1207--1227.


\bibitem{berf}
{A. Bertram and B. Feinberg},
\lq On stable rank two bundles with canonical determinant and many sections\rq,
{\em In Algebraic Geometry (Catania, 1993/Barcelona, 1994), 259--269, Lecture Notes in
Pure and Appl. Math.}, {\bf 200}, Dekker, New York, 1998.



\bibitem{bbn}
{ U. N. Bhosle, L. Brambila-Paz \and P. E. Newstead,}
\lq On linear systems and a conjecture of D. C. Butler\rq,
{\em Internat. J. Math.} {\bf 26} (2015), no. 1550007, 18 pages, doi: 10.1142/S0129167X1550007X.

%

\bibitem{bis}
{I. Biswas, L. Brambila-Paz and P. E. Newstead,}
\lq Stability of projective Poincar\'e and Picard bundles\rq,
{\em Bull. London Math. Soc.} {\bf 41} (2009), 458--472.

\bibitem{bgmn}
{ S. Bradlow, O. Garc\'ia-Prada, V. Mu\~noz \and P. E. Newstead,}
\lq Coherent systems and Brill-Noether Theory\rq,
{\em Internat. J. Math.} {\bf 14} (2003), 683--733.
%

\bibitem{bgmmn}
{ S. B. Bradlow, O. Garc\'ia-Prada, V. Mercat, V. Mu\~noz and P. E. Newstead,}
\lq On the geometry of moduli spaces of coherent systems on algebraic curves\rq,
{\em  Internat. J. Math.} {\bf 18} (2004), 411--453.

\bibitem{bgmmn2}
{ S. B. Bradlow, O. Garc\'ia-Prada, V. Mercat, V. Mu\~noz and P. E. Newstead,}
\lq Moduli spaces of coherent systems of  small slope on algebraic curves\rq,
{\em  Comm. in Algebra} {\bf 37} (2009), 2649--2678.

\bibitem{mio}
{ L. Brambila-Paz,}
\lq Non-emptiness of moduli spaces of coherent systems\rq,
 {\em Internat. J. Math.} {\bf 18}, no. 7 (2008), 777--799.


\bibitem{bgn}
{L. Brambila-Paz, I. Grzegorczyk and P. E. Newstead,}
\lq Geography of Brill-Noether loci for small slopes\rq,
{\em  J. Algebraic Geometry,} {\bf 6} (1997), 645--669.






\bibitem{br}
{S. Brivio,}
\lq On the degeneracy locus of a map of vector bundles on Grassmannian varieties \rq,
{\em Math. Nachr.} {\bf 244} (2002), 26--37.

\bibitem{bu}
{ D. C. Butler},
 \lq Normal generation of vector bundles over a curve\rq,
{\em J. Differential Geom.} {\bf 39}, no. 1 (1994), 1--34.

\bibitem{but}
{ D. C. Butler},
 \lq  Birational maps of moduli of Brill-Noether pairs\rq,
 preprint, arXiv:alg-geom/9705009.


\bibitem{cf1}
C. Ciliberto and F. Flamini,
\lq Brill-Noether loci of stable rank-two vector bundles on a general curve\rq,
{\em Geometry and Aritmetic}, ed. C.~Faber, G.~Farkas and R.~de Jong, EMS Series of Congress Reports, 61--74, European Mathematical Society, Z\"urich, 2012.

\bibitem{cf2}
C. Ciliberto and F. Flamini,
\lq Extensions of line bundles and Brill-Noether loci of rank-two vector bundles on a general curve\rq,
{\em Revue Roumaine de Math. Pures et Appliqu\'es} {\bf60} (2015), 201--255.

\bibitem{el} {L. Ein and R. Lazarsfeld,}
\lq Stability and restrictions of Picard bundles with an application to the normal bundles of elliptic curves\rq, {\em Complex projective geometry (Trieste 1989/Bergen 1989)},
ed. G.~Ellingsrud, C.~Peskine, G.~Sacchiero and S.~A.~Stromme, LMS Lecture Notes Series {\bf 179}, 149--156, CUP, Cambridge, 1992.

\bibitem{f}
G. Farkas,
\lq Progress on syzygies of algebraic curves\rq,
arXiv 1703.08056.

\bibitem{fo}
{ G. Farkas and A. Ortega,}
\lq The maximal rank conjecture and rank two Brill-Noether theory\rq,
{\em Pure and Applied Mathematics Quarterly} {\bf 7}, (2011), 1265--1296.


\bibitem{gmn}
{I. Grzegorczyk, V. Mercat and P. E. Newstead,}
\lq Stable bundles of rank 2 with 4 sections\rq,
{\em Internat. J. Math.} {\bf 22} (2011), 1743--1762.




\bibitem{kn}
{A. King and P.  E. Newstead,}
\lq Moduli of Brill-Noether pairs on algebraic curves\rq,
{\em Internat. J. Math.} {\bf 6} (1995), 733--748.


\bibitem{lnell}{H. Lange and P. E. Newstead,}
\lq Coherent systems on elliptic curves \rq,
{\em Internat. J. Math.} {\bf 16} (2005), 787--805.

\bibitem{nl}
{ H. Lange and P. E. Newstead,}
\lq Clifford's theorem for coherent systems\rq,
{\em Archiv der  Mathematik} {\bf 90} (2008) 3, 209--216.

\bibitem{lnci}
{H. Lange and P. E. Newstead,}
\lq Clifford indices for vector bundles on curves\rq,
 in Alexander Schmitt (ed.), Affine Flag Manifolds and Principal Bundles, Trends in Mathematics, Birkh\"auser Basel, 2010, 165--202.

\bibitem{lnrk2}
{H. Lange and P. E. Newstead},
\lq Vector bundles of rank 2 computing Clifford indices\rq,
{\em Comm. in Algebra} {\bf 41:6} (2013), 2317--2345.

\bibitem{lnrk3}
{H. Lange and P. E. Newstead},
\lq On bundles of rank 3 computing Clifford indices\rq,
{\em Kyoto J. Math.} {\bf 53}, no. 1 (Memorial issue for Masaki Maruyama) (2013), 25--54.

\bibitem{ln4}
{H. Lange and P. E. Newstead},
\lq Higher rank BN-theory for curves of genus 4\rq,
{\em Comm. in Algebra}, {\bf 45:9} (2017), 691--717, doi: 10.1080/00927872.2016.1251938.

\bibitem{ln6}
{H. Lange and P. E. Newstead},
\lq Higher rank Brill-Noether theory for curves of genus 6\rq,
arXiv 1703.09024.

\bibitem{li} {Y. Li},
\lq Spectral curves, theta divisors and Picard bundles\rq, {\em Internat. J. Math.} {\bf 2} (1991), 525--550.





\bibitem{mer}{V. Mercat,}
\lq Le probl\`eme de Brill-Noether pour des fibr\'es stables de petite pente\rq,
{\em J. Reine Angew. Math.} {\bf 506} (1999), 1--41.

\bibitem{mer2}{V. Mercat,}
\lq Fibr\'e stables de pente 2\rq,
{\em Bull. London Math. Soc. } {\bf 33} (2001), 535--542.





\bibitem{news}
{ P. E. Newstead,}
\lq Existence of $\alpha$- stable coherent systems on algebraic curves\rq,
{\em Clay Math. Proccedings} {\bf 14},  (2011), 121--139.


\bibitem{news1}
{ P. E. Newstead,}
\lq
 Some examples of rank-2 Brill-Noether loci\rq,
{\em Rev. Mat. Complutense}, 2017, doi:10.1007/s13163-017-0241-6 (open access).


\bibitem{o}
{B. Osserman},
\lq Brill-Noether loci with fixed determinant in rank 2\rq,
{\em Internat. J. Math}, {\bf 24} (2013), no. 1350099, 24 pages.

\bibitem{pr}
{K. Paranjape and S. Ramanan},
\lq On the canonical ring of a curve\rq,
{\em Algebraic Geometry and Commutative Algebra in Honor of Masayoshi Nagata} (1987), 503--516.

\bibitem{rv}
{N. Raghavendra and P. A. Vishwanath},
\lq Moduli of pairs and generalized theta divisors\rq,
{\em Tohoku Math. J.} {\bf 46} (1994), 321--340.



\bibitem{mon1}
M. Teixidor i Bigas,
\lq Brill-Noether theory for stable vector bundles\rq,
{\em Duke Math. J.} {\bf62} (1991), 385--400.

\bibitem{mon2}
{M. Teixidor i Bigas},
\lq Curves in Grassmannians\rq,
{\em Proc. Amer. Math. Soc.} {\bf 126} (1998), no. 6, 1597--1603.

\bibitem{T2}
{M. Teixidor i Bigas},
\lq Existence of coherent systems of rank two and dimension four\rq,
{\em Collectanea Math.}, 58 (2007), 193--198.

\bibitem{T3}
{M. Teixidor i Bigas},
\lq Syzygies using vector bundles\rq,
{\em Trans. Amer. Math. Soc.}, {\bf 359} (2007), no.2, 897--908.

\bibitem{mon}
{M. Teixidor i Bigas}
\lq Existence of coherent  systems II\rq,
{\em Int. J. Math.} {\bf 19,} 449 (2008), 1269--1283.





\end{thebibliography}
\end{document}